\documentclass{article}
\usepackage{amsmath}
\usepackage{amsthm}
\usepackage{amssymb,epsfig}
\usepackage[cp 1250]{inputenc}
\usepackage{graphics}
\usepackage[IL2]{fontenc}

\usepackage{cleveref}
\newtheorem{defi}{Definition}
\newtheorem{ex}{Example}
\newtheorem{thm}{Theorem}
\newtheorem{lemma}[thm]{Lemma}

\newtheorem{ot}{Question}

\newtheorem{poz}{Observation}
\newtheorem{prop}{Proposition}

\begin{document}

\title{Antipalindromic numbers}

\author{Lubom\' ira Dvo\v r\'akov\'a~\footnote{Faculty of Nuclear Sciences and Physical Engineering, Czech Technical University in Prague, Czech Republic, email: lubomira.dvorakova@fjfi.cvut.cz}, Stanislav Kruml~\footnote{Faculty of Nuclear Sciences and Physical Engineering, Czech Technical University in Prague, Czech Republic}, David Ryz\'ak~\footnote{Faculty of Mathematics and Physics, Charles University, Prague, Czech Republic}}
\date{}
\maketitle

\begin{abstract}
Everybody has certainly heard about palindromes: words that stay the same when read backwards. For instance kayak, radar, or rotor.
Mathematicians are interested in palindromic numbers: positive integers whose expansion in a certain integer base is a palindrome. The following problems are studied:  palindromic primes, palindromic squares and higher powers, multi-base palindromic numbers, etc. In this paper, we define and study antipalindromic numbers: positive integers whose expansion in a certain integer base is an antipalindrome. We present new results concerning divisibility and antipalindromic primes, antipalindromic squares and higher powers, and multi-base antipalindromic numbers. We provide a~user-friendly application for all studied questions.
\end{abstract}

\section{Introduction}

Everybody has certainly heard about palindromes: words that stay the same when read backwards. For instance kayak, radar, or rotor.
It is not surprising that in natural languages, it is impossible to find extremely long palindromes.
The longest palindrome in English is ``tattarrattat''. However its victory is doubtful since tattarrattat is a neologism created by
James Joyce in his novel Ulysses~\cite{Jo}; it expresses loud knocking at the door:

\begin{center}
``I was just beginning to yawn with nerves thinking he was trying to
make a~fool of me when I knew his tattarrattat at the door.''
\end{center}
Palindromic phrases are even more interesting. They provide palindromes if punctuation, capitalization, and spaces are ignored.
Some popular palindromic phrases in English are:
\begin{center}
\begin{tabular}{c}
``Do geese see God?''\\
``A man, a plan, a canal: Panama. ''\\
``Madam, in Eden, I’m Adam.''
\end{tabular}
\end{center}
Mathematicians are interested in palindromic numbers: positive integers whose expansion in a certain integer base is a palindrome.
Let us make a list of studied problems:
\begin{enumerate}
\item {\em Palindromic squares, cubes, and higher powers in base} $10$: The first nine terms of the sequence $1^2, {11}^2, {111}^2, {1111}^2, \dots$ are palindromic numbers $1, 121, 12321, 1234321, \dots$ (sequence A002477 in the OEIS~\cite{OEIS}). The only known non-palindromic number whose cube is a palindromic number is $2201$, and Simmons~\cite{Square2} conjectured that the fourth root of all palindromic fourth powers are palindromic numbers of the form ${10}^n + 1$. Simmons~\cite{Square1} also conjectured there are no palindromic numbers of the form $n^k$ for $k > 4$ and $n>1$.
\item {\em Palindromic primes}: The first few decimal palindromic primes are (sequence A002385 in the OEIS~\cite{OEIS}):
$$2, 3, 5, 7, 11, 101, 131, 151, 181, 191, 313, 353,$$
$$373, 383, 727, 757, 787, 797, 919, 929, \dots$$
Except for $11$, all palindromic primes have an odd number of digits because the divisibility test for $11$ tells us that every palindromic number with an even number of digits is divisible by $11$.
On one hand, it is not known if there are infinitely many palindromic primes in base $10$; the largest known decimal palindromic prime has 474,501 digits (found in 2014):
$$10^{474500} + 999 \cdot 10^{237249} + 1.$$
On the other hand, it is known that, for any base, almost all palindromic numbers are composite~\cite{Ba}. It means the ratio of palindromic composites and all palindromic numbers less than $n$ tends to $1$.

Binary palindromic primes include the Mersenne primes and the Fermat primes~\footnote{A Mersenne prime is a prime of the form $2^p-1$, where $p$ is a prime. A Fermat prime is a prime of the form $2^{2^n}+1$.}. All binary palindromic primes except the number $3$ (having the expansion $11$ in base $2$) have an odd number of digits; palindromic numbers with an even number of digits are divisible by $3$. Let us write down the sequence of binary expansions of the first binary palindromic primes (sequence A117697 in the OEIS~\cite{OEIS}):
$$11, 101, 111, 10001, 11111, 1001001, 1101011,$$
$$1111111, 100000001, 100111001, 110111011,\dots$$
\item {\em Multi-base palindromic numbers}: Any positive integer $n$ is palindromic in all bases $b$ with $b \geq n + 1$ because $n$ is then a single-digit number, and also in base $n-1$ because the expansion of $n$ in base $n-1$ equals $11$.
    But, it is more interesting to consider bases smaller than the number itself. For instance the number $105$ is palindromic in bases $4, 8, 14, 20, 34, 104$; the expansions of $105$ in those bases are:
    $$(105)_4=1221, \quad (105)_8=151, \quad (105)_{14}=77,$$
    $$(105)_{20}=55,\quad (105)_{34}=33, \quad (105)_{104}=11.$$
    A palindromic number in base $b$ whose expansion is made up of palindromic sequences of length $\ell$ arranged in a palindromic order is palindromic in base $b^{\ell}$. For example, the number $24253$ has the expansion in base $2$ equal to $(24253)_2=101\ 111\ 010\ 111\ 101$, i.e., it is made up of palindromes of length $3$, and its expansion in base $2^3=8$ is equal to $(24253)_8=57275$.
\item {\em Sum of palindromes}: Every positive integer can be written as the sum of at most three palindromic numbers in every number system with base $5$ or greater~\cite{Ci}.
    \end{enumerate}

In this paper, we deal with antipalindromic numbers in various integer bases. We examine and compare properties of palindromic numbers and antipalindromic numbers, bringing a number of new results. These are structured as follows. In Section \ref{sec:Definice}, we introduce the definition of an antipalindromic number in an integer base and its basic properties, following from the definition. Section \ref{sec:Primes} brings surprising results concerning divisibility and antipalindromic primes. In Section~\ref{sec:powers}, antipalindromic squares and higher powers are examined. Section~\ref{sec:Multibased} contains information about numbers that are antipalindromic in two or more bases at the same time. In Section~\ref{sec:OpenProblems}, we summarize our results and provide a list of conjectures and open problems.

\section{Definition and basic properties}\label{sec:Definice}
Let us start with a formal definition of palindromic and antipalindromic numbers and their basic properties.

\begin{defi}\label{def:pal_antipal}
Let $b \in \mathbb N, b \geq 2$. Consider a natural number $m$ whose expansion in base $b$ is of the following form
$$m=a_{n}b^{n}+\dots+a_{1}b+a_{0},$$
where $a_0, a_1, \dots, a_n \in \{0, 1, \dots, b-1\}, a_n \not =0$. We usually write $(m)_b=a_{n}\dots a_{1}a_{0}$.
Then $m$ is called
\begin{enumerate}
\item {\em a palindromic number} in base $b$ if its digits satisfy the condition:
\begin{equation}\label{eq:palindrom}
a_j=a_{n-j} \quad \text{for all $j \in \{0, 1, \dots, n\}$,}
\end{equation}
\item {\em an antipalindromic number} in base $b$ if its digits satisfy the condition:
\begin{equation}\label{eq:antipalindrom}
a_j=b-1-a_{n-j} \quad \text{for all $j \in \{0, 1, \dots, n\}$.}
\end{equation}
The length of the expansion of the number $ m $ is usually denoted $|m|$.
\end{enumerate}
\end{defi}

\begin{ex}
Consider distinct bases $ b $ and have a~look at antipalindromic numbers in these bases:
\begin{itemize}
\item $395406$ is an antipalindromic number in base $ b=10$.
\item $(1581)_3 = 2011120$, i.e., 1581 is an antipalindromic number in base $b=3$.
\item $(52)_2 = 110100$, i.e., 52 is an antipalindromic number in base $b=2$.
\end{itemize}
\end{ex}

\begin{prop}\label{thm:antipal_lichy_pocet_cifer}
If an antipalindromic number in base $b$ has an odd number of digits, then $b$ is an odd number and the middle digit is equal to $\frac{b-1}{2}$.
\end{prop}

\begin{proof}
Let us denote the digits of the considered antipalindromic number \\ $a_{0}, a_{1}, a_{2}, \dots, a_{2n}$. Pair the digits and add $ a_{0}+a_{2n}, a_{1}+a_{2n-1},\dots ,  a_{n-1}+a_{n+1}$. From the definition, each pair has a total of $b-1$. That leaves us with the digit $a_{n}$ that must be paired with itself: $ 2 a_n =b-1$.
\noindent Therefore, the digit $a_{n}$ is an integer only for $b=2k+1$, where ${k}\in\mathbb{N} $, i.e., for an odd $b$. Furthermore, $a_{n}=\frac{b-1}{2} $.
\end{proof}

\begin{prop}
A number is simultaneously palindromic and antipalindromic if and only if $b$ is an odd number and all the digits are equal to $ \frac{b-1}{2} $.
\end{prop}

\begin{proof}
Consider an antipalindromic number with digits $ a_{0}, a_{1},\dots, a_{n}$. For this number to be palindromic, $ a_{j}=a_{n-j}$ must be true for each $j \in \{0, 1,  \dots, n\}$. From the definition of an antipalindromic number, it follows $ a_{j}+a_{n-j}=b-1$. For each $j \in \{0, 1 ,\dots , n\}$, we obtain $a_{j}=\frac{b-1}{2}$, i.e., all digits are equal to $\frac{b-1}{2}$ and the base $b$ must therefore be odd. The opposite implication is obvious.
\end{proof}

\section{Divisibility and antipalindromic primes}\label{sec:Primes}
Let us first study divisibility of antipalindromic numbers, which will be used in the sequel to show surprising results on antipalindromic primes.

\begin{lemma}\label{lem4}
Let $m$ be a natural number and its expansion in base $b$ be equal to $a_{n}b^n+a_{n-1}b^{n-1}+\ldots +a_{1}b+a_0$. Then $m$ is divisible by $b-1$ if and only if the sum of its digits is divisible by $b-1$, i.e., $a_n+a_{n-1}+\ldots +a_1+a_0 \equiv 0  \ (\mathrm{mod}\ b-1)$.
\end{lemma}
\begin{proof}
The statement follows from the fact  \mbox{$b^k\equiv 1 \ (\mathrm{mod}\ b-1)$ for any $k \in \mathbb N$}.
\end{proof}

\begin{thm}\label{diviseven}
Any antipalindromic number with an even number of digits in base $b$ is divisible by $b-1$.
\end{thm}

\begin{proof}
Consider an antipalindromic number
\begin{equation*}
    m = a_{n}b^n+a_{n-1}b^{n-1}+\ldots +a_{1}b+a_0
\end{equation*}
for an odd $n$. From the definition, it is true that $a_j+a_{n-j}=b-1$ for each $j\in \{ 0,1,\ldots ,n\}$. The number of digits is even, hence
\begin{equation*}
    a_n+a_{n-1}+\ldots +a_1+a_0 = (b-1)\frac{n+1}{2}\equiv 0\  (\mathrm{mod}\ b-1).
\end{equation*}
Using Lemma \ref{lem4}, the antipalindromic number $m$ is divisible by $b-1$.
\end{proof}

\begin{thm}\label{divisodd}
An antipalindromic number with an odd number of digits in base $b$ is divisible by $\frac{b-1}{2}$.
\end{thm}

\begin{proof}
Consider the antipalindromic number
\begin{equation*}
    m = a_{2n}b^{2n} + a_{2n-1}b^{2n-1} + \ldots + a_{1}b+a_0.
\end{equation*}
The digit sum of the number $m - a_{n}b^n$ is divisible by $b-1$. From Lemma \ref{lem4}, we also know that the number $m - a_{n}b^n$ itself is divisible by $b-1$ and, therefore, by $\frac{b-1}{2}$. From the definition, $a_{n}=\frac{b-1}{2}$. The number $m$ is a sum of two numbers divisible by $\frac{b-1}{2}$.
\end{proof}

Let us now turn our attention to antipalindromic primes. While palindromic primes occur in various bases, antipalindromic primes occur (except some trivial cases) only in base 3.

\begin{thm}
\label{jednoPrvocislo1}
Let base $b>3$. Then there exists at most one antipalindromic prime number $p$ in base $b$: $p=\frac{b-1}{2}$.
\end{thm}

\begin{proof}
Theorems \ref{diviseven} and \ref{divisodd} show that every antipalindromic number is divisible either by $\frac{b-1}{2}$ or $b-1$. Although $b-1$ may be a prime number, it is never antipalindromic.
\end{proof}

\begin{thm}
\label{jednoPrvocislo2}
Let base $b=2$. Then there exists only one antipalindromic number $p = 2$, $(p)_2=10$.
\end{thm}

\begin{proof}
Every antipalindromic number in base $b=2$ is even. $2$ is the only even prime number.
\end{proof}

\begin{thm}
Let base $b=3$. Every antipalindromic prime in this base has an odd number of digits $n\geq 3$.
\end{thm}

\begin{proof}
From Theorem \ref{diviseven}, antipalindromic numbers with an even number of digits in base $b=3$ are even. The only antipalindromic number in this base with one digit is 1.
\end{proof}

\begin{lemma}\label{lem:div3}
Antipalindromic numbers in base $b=3$ beginning with a digit $2$ are divisible by $3$.
\end{lemma}

\begin{proof}
Consider an antipalindromic number $m=a_n3^n+a_{n-1}3^{n-1}+\ldots+a_13+a_0$, where $a_n=2$. The sum of $a_n$ and $a_0$ need to be equal to $2$, therefore $a_0=0$. All the summands are divisible by 3.
\end{proof}

\begin{thm}
All antipalindromic primes in base $b=3$ can be expressed as $6k+1$, where $k \in \mathbb N$.
\end{thm}

\begin{proof}
Consider an antipalindromic prime $m=a_{2n}3^{2n}+a_{2n-1}3^{2n-1}+\ldots +a_13+a_0$. (The number of digits must be odd.) From Lemma \ref{lem:div3}, $a_0$ is equal to $1$. Let us pair the digits of the antipalindromic number $m$ (except for $a_{2n}$, $a_n$, and $a_0$): $a_{2n-j}3^{2n-j} + a_j3^j, j \in \{1,\ldots ,n-1\}.$ Let us prove that for each $j\in \{1,\ldots ,n-1\}$ there exists $s \in \mathbb N$ satisfying
\begin{equation*}
    3^j( a_{2n-j}3^{2n-2j} + a_j) = 6s.
\end{equation*}
We can only consider three possibilities: $a_{2n-j}=2,a_j=0$, or $a_{2n-j}=a_j=1$, or $a_{2n-j}=0,a_j=2$. In either case, the equation holds because there is an even number inside the bracket. We then get
\begin{align*}
m\  &=\  a_{2n}3^{2n}+a_n3^n+a_0+6\ell\\
\ &= \ 3^{2n}+3^n+1+6\ell\\
\ &= \ 3^n(3^n+1)+1+6\ell
\end{align*}
for some $\ell \in \mathbb N$. The first summand is also divisible by $6$, therefore $m$ can indeed be expressed as $6k+1$ for some $k\in \mathbb N$.

\end{proof}

The application~\cite{Kr} can be used for searching antipalindromic primes in base 3. During an extended search, the first 637807 antipalindromic primes have been found. Let us now list at least the first 10 of them, along with their expansions in base 3:
\begin{align*}
    13&\  &\ 111\\
    97&\  &\ 10121\\
    853&\  &\ 1011121\\
    1021&\  &\ 1101211\\
    1093&\  &\ 1111111\\
    7873&\  &\ 101210121\\
    8161&\  &\ 102012021\\
    8377&\  &\ 102111021\\
    9337&\  &\ 110210211\\
    12241&\  &\ 121210101
\end{align*}

\section{Squares and other powers as antipalindromes}\label{sec:powers}

For palindromic numbers, squares and higher powers were considered in \cite{Square1,Square2} by Simmons more than thirty years ago. He proved that there were infinitely many palindromic squares, cubes, and biquadrates. However, his conjecture was that for $k>4, k\in \mathbb N$, no integer $m$ exists such that $m^k$ is a palindromic number (in the decimal base). This conjecture is still open. That is definitely not the case for antipalindromic numbers as $3^7=2187$ is antipalindromic in base 10.\par
Let us answer the following question:
\begin{ot}
Are there any antipalindromic integer squares?
\end{ot}

Our initial observation suggested that bases $b=n^2+1,n\in \mathbb N$, have the most antipalindromic squares and the computer application provided additional insight needed to prove this observation not only for squares but for other powers as well. Table~\ref{table:squares} expresses the number of antipalindromic squares smaller than $10^{12}$ in bases $n^2, n^2+1,$ and $n^2+2$ to underline the differences between the bases of the form $n^2+1$ and the others.

\begin{table}[ht]
\centering
\begin{tabular}{|c|cccccc|}

\hline
base                   & n=20 & 21 & 22 & 23 & 24 & 25  \\ \hline
$n^2$   & 3    & 13 & 3  & 14 & 9  & 11   \\
$n^2+1$ & 47   & 44 & 48 & 53 & 55 & 57  \\
$n^2+2$ & 2    & 2  & 2  & 2  & 2  & 1  \\ \hline
\end{tabular}
\caption{Number of antipalindromic squares smaller than $10^{12}$ in particular bases}\label{table:squares}
\end{table}

As the exponent is raised, the differences become even more significant, but the numbers rise faster, see Table~\ref{table:biquadrates}.
\begin{table}[ht]
\centering
\begin{tabular}{|c|ccccccccc|}

\hline
base                   & n=4 & 5 & 6 & 7 & 8 & 9 & 10&11&12 \\ \hline
$n^4$   & 0    & 1 & 0  & 1 & 0  & 1 & 0&0&0  \\
$n^4+1$ & 6   & 6 & 8 & 10 & 13 & 13 & 13&13&13 \\
$n^4+2$ & 0    & 0  & 0  & 0  & 0  & 0  & 0 &0&0\\ \hline
\end{tabular}
\caption{Number of antipalindromic biquadrates smaller than $10^{15}$ in particular bases}\label{table:biquadrates}
\end{table}

\begin{ex} Consider the base $b=10=3^2+1$. Any antipalindromic number in this base must be divisible by 9.
Every double-digit number divisible by 9 (except 99) is antipalindromic:
\begin{equation*}
18,27,36,45,54,63,72,81,90.
\end{equation*}
The number 9 is a square, so if a square is divided by 9, it still is a square. \\
\begin{equation*}
36=4\cdot 9 = 2^2\cdot 3^2 = 6^2,\quad 81=9\cdot 9 = 9^2.
\end{equation*}
Thus $36$ and $81$ are antipalindromic squares.
\end{ex}

\begin{prop}
For $b=n^2+1,  n\in \mathbb N$, and $m \in \{ 2,3,\ldots ,n\} $, the number $(m\cdot n)^2$ is antipalindromic.
\end{prop}

\begin{proof}

Since $b=n^2+1$, we can modify the expression as follows: \\ $(m\cdot n)^2 = m^2\cdot (b-1)$.
This number has the expansion in base $b$ equal to $(m^2-1)\ (b-m^2)$, hence it is antipalindromic.
\end{proof}

\begin{ot}
Are there any higher integer powers that are also antipalindromic numbers?
\end{ot}

\begin{ex} Consider the base $b=28=3^3+1$. Any antipalindromic number in this base with an even number of digits must be divisible by 27.
Every double-digit number divisible by 27 (except the one with expansion $(27)(27)$) is antipalindromic.

The number 27 is a third power of 3, so if a third power of any number is divided by 27, it still is a third power of an integer.
$$(7) (20) = (216)_{28}\ \text{and} \ 216 = 8\cdot 27 = 2^3 \cdot 3^3 = 6^3,$$
$$(26) (1) = (729)_{28} \ \text{and} \ 729= 27\cdot 27 = 3^3\cdot 3^3 = 9^3.$$

Thus $216$ and $729$ are antipalindromic cubes.
\end{ex}

\begin{thm}
\label{nk+1}
For $b=n^k+1$, where $n,k\in \mathbb N, k\geq 2$, and $m \in \{2,3,\ldots ,n\}$, the number $(m\cdot n)^k$ is antipalindromic.
\end{thm}

\begin{proof}
Since $b=n^k+1$, we can modify the expression as follows: \\ $(m\cdot n)^k = m^k\cdot (b-1)$.
This number has the expansion in base $b$ equal to $(m^k-1) (b-m^k)$, thus it is antipalindromic.
\end{proof}
For odd powers and high enough bases, other patterns exist.

%\begin{lemma}
% For each $m > 1$, there exists a number $c$ such that in every base $b> c$, the following number is antipalindromic:
%\begin{equation*}
%[m\cdot (b-1)]^3.
%\end{equation*}
%It suffices to put  $c=3\cdot m^3 - 1$.
%\end{lemma}
%
%\begin{proof} Let us write down the expansion of $[m\cdot (b-1)]^3$:
%$$([m\cdot(b-1)]^3)_b  =   (m^3-1) \quad (b-3\cdot m^3)\quad (3\cdot m^3-1)\quad (b-m^3).$$
%\end{proof}

\begin{thm}
 For integers $m>1$ and odd $k>1$, there exists a number $c$ such that in every base $b\geq c$, the following number is antipalindromic:
 \begin{equation*}
     [m\cdot (b-1)]^k.
 \end{equation*}
 It suffices to put $c = \binom{k}{\frac{k-1}{2}}\cdot m^k$.
\end{thm}

\begin{proof}
The binomial theorem reads
$$[m\cdot (b-1)]^k = m^k\cdot\sum_{i=0}^{k}(-1)^i\cdot \binom{k}{i}\cdot b^{k-i}. $$\\
Since $\binom{k}{\frac{k-1}{2}}$ is the maximum number  among $\binom{k}{i}$ for $i \in \{0,1,\dots ,k\}$, the expansion in base $b$ equals:
$$\left([m\cdot (b-1)]^k\right)_b =
\left(m^k\cdot\binom{k}{0}-1\right) \left(b-m^k\cdot\binom{k}{1}\right)\dots$$
$$\dots\left(m^k\cdot\binom{k}{k-1}-1\right)
\left(b-m^k\cdot\binom{k}{k}\right).$$
\end{proof}

\section{Multi-base antipalindromic numbers}\label{sec:Multibased}
Let us study the question whether there are numbers that are antipalindromic simultaneously in more bases. In his 2010 paper~\cite{Basic}, Ba\v si\'c showed that for any list of bases, there exists a number with palindromic expansions in each of the bases. This does not necessarily apply to antipalindromic numbers as there exist sets of bases (e.g., 6 and 8) for which our application mode for searching multi-base antipalindromic numbers was unable to find any simultaneous antipalindromic numbers.\par

In 2014, B\'erczes and Ziegler~\cite{Simul} discussed multi-base palindromic numbers and proposed a list of the first 53 numbers palindromic in bases $2$ and $10$ simultaneously. Our application~\cite{Kr} has only been able to find one number with an antipalindromic expansion in these bases. This number, $3276$, is also antipalindromic in other 19 distinct bases, see Table~\ref{table:multibase}. The next greater number that is antipalindromic both in base $2$ and $10$ must be greater than $10^{10}$ and divisible by 18. \par

It is not uncommon for a number to be antipalindromic in more bases. In this section, we show that if a number is antipalindromic in a unique base, then the number must be prime or equal to 1, see Theorem \ref{multithm}. \\

\begin{table}[ht]
\centering
\begin{tabular}{|c|c|}

\hline
base                   & expansion \\ \hline
$2$ & 110011001100 \\
$4$ & 303030 \\
$10$ & 3276\\
$64$ & 53 10\\
$79$ & 41 37\\
$85$ & 38 46\\
$92$ & 35 56\\
$118$ & 27 90\\
$127$ & 25 101\\
$157$ & 20 136\\
$183$ & 17 165\\
$235$ & 13 221\\
$253$ & 12 240\\
$274$ & 11 262\\
$365$ & 8 356\\
$469$ & 6 462\\
$547$ & 5 541\\
$820$ & 3 816\\
$1093$ & 2 1090\\
$1639$ & 1 1637\\
$6553$ & 3276\\ \hline
\end{tabular}

\caption{Antipalindromic expansions of the number $3276$ in $21$ bases}\label{table:multibase}
\end{table}

\begin{defi} An antipalindromic number is called {\em multi-base} if it is antipalindromic in at least two different bases.
\end{defi}

\begin{poz}
Every number $m \in \mathbb N$ is antipalindromic in base $2m+1$.
\end{poz}

\begin{ex}
The number 3276 is a multi-base antipalindromic number, as illustrated in Table~\ref{table:multibase}.
\end{ex}

\begin{thm}\label{multithm}For any composite number $a\in \mathbb N$, we can find at least two bases $b,c$ such that this number has an antipalindromic expansion in both of them.
\end{thm}

\begin{proof}
Assume that $a= m\cdot n,\ m,n \in \mathbb N, \  m\geq n\geq 2,$
set $b= \frac{a}{n}+1, \
c = 2a+1$.
The expansions of $a$ in bases $b,c$ are equal to:
\begin{align*}
    (a)_{b}&= (n-1) (\tfrac{a}{n}-n+1),\\
    (a)_{c}&= a.
\end{align*}
\end{proof}

\begin{thm}
For every $n\in \mathbb N$, there exist infinitely many numbers that are antipalindromic in at least $n$ bases.
\end{thm}

\begin{proof}
Consider a number $a$ such that $a = (2n)!$. Theorem \ref{multithm} indicates that the number $a$ is antipalindromic in bases $\frac{a}{2}+1,\frac{a}{3}+1,\ldots,\frac{a}{n}+1$ and also $2a+1$.
\end{proof}

\begin{thm}
Let $b\in \mathbb N, b\geq 2$. Then there exists $m\in \mathbb N$ such that $m$ is antipalindromic in base $b$ and in at least one more base less than m.
\end{thm}

\begin{proof}
$$\begin{array}{lll}
\text{base}\ b & m& \text{two expansions}\\

2  & 12    & (12)_2=1100  \\
  &     &  (12)_4=30  \\
3 & 72   & (72)_3=2200\\
 &    & (72)_9=80\\
\geq 4 & 4\cdot (b-1)    & (m)_b=3\ (b-4)\\
&    & (m)_{2b-1}=1\ (2b-3)\\
\end{array}
$$

\end{proof}

\begin{thm}
Let $p,q\in \mathbb N$ such that $\gcd (p,q)=d$, $p=p'\cdot d$, $q=q'\cdot d$ and $p\geq q'>1$, $q\geq p'>1$. Then the number $m=p'\cdot q'\cdot d = p\cdot q' = q\cdot p'$ is antipalindromic in bases $p+1$ and $q+1$.
\end{thm}

\begin{proof}
We have
\begin{align*}
(m)_{p+1}&=(q'-1)(p+1-q'),\\
(m)_{q+1}&=(p'-1)(q+1-p').
\end{align*}
\end{proof}

\begin{ex}
Let $p=4$, $q=6$, then $\gcd(4,6)=2$. \\ The number $m=12$ is antipalindromic in bases $5$ and $7$:  $(12)_5=22$, $(12)_7=15$.
\end{ex}

We have mentioned in the introduction part that a palindromic number in base $b$ whose expansion is made up of palindromic sequences of length $\ell$ arranged in a palindromic order is palindromic in base $b^{\ell}$. Let us present a similar statement for antipalindromic numbers.
For its proof, we will need the following definition.
\begin{defi}
Let $b\in \mathbb N, b\geq 2$.
Consider a string $u = u_0u_1\ldots u_n$, where $u_i \in \{0,1,\ldots, b-1\}$. The {\em antipalindromic complement} of $u$ in base $b$ is \mbox{$A(u) = (b-1-u_n)(b-1-u_{n-1})\ldots (b-1-u_0)$.}
\end{defi}
\begin{thm}\label{thm:multi-base}
Let $b \in \mathbb N, b \geq 2$.
An antipalindromic number $m$ in base $b^n$, where $(m)_{b^n}=u_k \dots u_1 u_0$ and $u_k\geq b^{n-1}$, is simultaneously antipalindromic in base $b$ if and only if the expansion of $u_j$ in base $b$ of length $n$ (i.e., completed with zeroes if necessary) is a palindrome for all $j \in \{0,1, \dots, k\}$.
\end{thm}

\begin{proof}
The digits of $m$ in base $b^n$ satisfy $0 \leq u_j \leq b^n-1$.
Let us denote the expansion of $u_j$ in base $b$ by $(u_j)_b=v_{j,{n-1}}\dots v_{j, 1}v_{j, 0}$ (where the expansion of $u_j$ in base $b$ is completed with zeroes in order to have the length $n$ if necessary). The antipalindromic complement $A(u_j)$ of $u_j$ in base $b^n$ equals $b^n-1-u_j$ and its expansion in base $b$ equals $(A(u_j))_b=(b-1-v_{j,n-1})\dots (b-1-v_{j,1})(b-1-v_{j,0})$.
Since $m$ is antipalindromic in base $b^n$, we have $u_{k-j}=A(u_j)=b^n-1-u_j$ for all $j \in \{0,1,\dots, k\}$.\\
Let us now consider the expansion of $m$ in base $b$: it is obtained by concatenation of the expansions of $u_j$ in base $b$ for $j\in \{0,1,\dots, k\}$, i.e.,
$$(m)_b=(u_{k})_b\dots (u_1)_b(u_0)_b.$$
Following the assertion that $u_k \geq b^{n-1}$, the expansion $(u_k)_b$ starts in a non-zero. Thus, the length of the expansion $(m)_b$ equals $n\cdot |(m)_{b^n}|$.\\

\noindent The number $m$ is antipalindromic in base $b$ if and only if $(u_{k-j})_b=A((u_j)_b)$ for all $j\in \{0,1,\dots, k\}$, i.e.,
$$\begin{array}{l}(b-1-v_{j,n-1})\dots (b-1-v_{j,1})(b-1-v_{j,0})=\\
=A(v_{j,{n-1}}\dots v_{j, 1}v_{j, 0})=\\
=(b-1-v_{j, 0})(b-1-v_{j, 1})\dots (b-1-v_{j,{n-1}}).
\end{array}$$
Consequently, $m$ is antipalindromic in base $b$ if and only if \\ $(u_j)_b=(v_{j,{n-1}}\dots v_{j, 1}v_{j, 0})$ is a palindrome for all $j \in \{0,1,\dots, k\}$.
\end{proof}

\begin{ex}
Consider $m=73652$. Then $(m)_{27}=3\ 20 \ 6 \ 23=u_3u_2u_1u_0$, thus $m$ is antipalindromic in base $27=3^3$. However, $(m)_3=10202020212$, thus $m$ is not antipalindromic in base $3$. If we cut $(m)_3$ into blocks of length 3, then all of them are palindromic. However, the first one equals $010$ and it starts in zero, hence the assumption $u_3 \geq 9$ of Theorem~\ref{thm:multi-base} is not met.
\end{ex}

\begin{ex}
Consider $b=10$. The number $6633442277556633$ is an antipalindromic number both in base $10$ and $100$.
\end{ex}

\section{Open problems }\label{sec:OpenProblems}
In this paper, we carried out a thorough study of antipalindromic numbers and described known results on palindromic numbers in order to draw a comparison.
It brings a number of new results:

\begin{itemize}
\item We described divisibility of antipalindromic numbers and showed that non-trivial antipalindromic primes may be found only in base 3.
\item We found several classes of antipalindromic squares and higher powers.
\item We described pairs of bases such that there is a number antipalindromic in both of these bases. Moreover, we obtained the following interesting results concerning multi-base antipalindromic numbers:
\begin{itemize}
\item For any composite number, there exist at least two bases such that this number is antipalindromic in both of them.
\item For every $n\in \mathbb N$, there exist infinitely many numbers that are antipalindromic in at least $n$ bases.
\item Let $b\in \mathbb N, b\geq 2$. Then there exists $m\in \mathbb N$ such that $m$ is antipalindromic in base $b$ and in at least one more base less than $m$.
\end{itemize}
\end{itemize}

This paper is based on the bachelor thesis~\cite{KrBP}, where some more results were obtained:
\begin{itemize}
\item the number of (anti)palindromic numbers of a certain length and the maximum and minimum number of antipalindromic numbers between palindromic numbers and vice versa;
\item an explicit formula for the length of gaps between neighboring antipalindromic numbers.
\end{itemize}

We created a user-friendly application for all the questions studied~\cite{Kr}, which is freely available to the reader. Based on computer experiments, we state the following conjectures and open problems:

\begin{enumerate}
\item
Are there infinitely many antipalindromic primes in base $3$? (We know there is never more than one antipalindromic prime in any other base except for $3$.) During an extended search, the first 637807 antipalindromic primes have been found.
\item
We conjecture it is possible to express any integer number (except for 24, 37, 49, 117, and 421) as the sum of at most three antipalindromic numbers in base 3. Our computer program shows that the answer is positive up to $5\cdot 10^6$.
\item
We conjecture it is possible to express any palindromic number in base 3 as the sum of at most three antipalindromic numbers in base 3. This conjecture follows evidently from the previous one, and we verified it even for larger numbers, up to $10^8$.
\item
Is there a pair of bases such that it is impossible to find any number that has an antipalindromic expansion in both of them? According to our computer experiments, suitable candidates seem to be the bases 6 and 8. It is to be studied in the future.
    \end{enumerate}

\section{Acknowledgements}
L.~Dvo\v r\'akov\'a received funding from the Ministry of Education, Youth and Sports of the Czech Republic through the project\\ no. CZ.02.1.01/0.0/0.0/16\_019/0000778.

\bibliographystyle{actapoly}
\bibliography{biblio}

\end{document}